\documentclass{amsart}
\usepackage{amsthm,amssymb,tikz,enumitem,color,hyperref}
\usetikzlibrary{arrows,matrix}

\newcommand{\linfty}{\ensuremath{\ell^\infty}}
\newcommand{\Linfty}{\ensuremath{L^\infty}}
\newcommand{\cat}[1]{\ensuremath{\mathbf{#1}}}
\newcommand{\cCstar}{\cat{cCstar}}
\newcommand{\Cstar}{\cat{Cstar}}
\newcommand{\Wstar}{\cat{Wstar}}
\newcommand{\AWstar}{\cat{AWstar}}
\newcommand{\proCstar}{\cat{proCstar}}
\newcommand{\Alg}{\cat{A}}
\newcommand{\TAlg}{\cat{TAlg}}
\newcommand{\Z}{\mathbb{Z}}
\newcommand{\C}{\mathbb{C}}
\newcommand{\T}{\mathbb{T}}
\newcommand{\N}{\mathbb{N}}
\newcommand{\M}{\mathbb{M}}
\newcommand{\cC}{\mathcal{C}}
\newcommand{\F}{\mathcal{F}}
\newcommand{\inprod}[2]{\ensuremath{\langle #1 , #2 \rangle}}

\DeclareMathOperator{\colim}{colim}
\DeclareMathOperator{\Spec}{Spec}
\newcommand{\ev}{\mathrm{ev}}
\newcommand{\fhat}{\widehat{f}}
\theoremstyle{theorem}
\newtheorem{theorem}{Theorem}
\newtheorem{lemma}[theorem]{Lemma}
\newtheorem{corollary}[theorem]{Corollary}
\newtheorem{proposition}[theorem]{Proposition}
\newtheorem{question}[theorem]{Question}
\theoremstyle{definition}
\newtheorem{definition}[theorem]{Definition}
\newtheorem{remark}[theorem]{Remark}
\newtheorem{example}[theorem]{Example}
\numberwithin{theorem}{section}
\numberwithin{equation}{section}

\begin{document}

\title{Discretization of C*-algebras}

\author{Chris Heunen}
\address{Department of Computer Science, University of Oxford, Oxford OX1 3QD, UK}
\email{heunen@cs.ox.ac.uk}
\author{Manuel L. Reyes}
\address{Department of Mathematics, Bowdoin College\\
Brunswick, ME 04011--8486, USA}
\email{reyes@bowdoin.edu}

%\thanks{
%  C.~Heunen was supported by EPSRC Fellowship EP/L002388/1. 
%  M.~Reyes was supported by NSF grant DMS-1407152. 
%  We thank the anonymous referee for an insightful report that sharpened the statement of
%  Theorem~\ref{thm:nogo}, informed us of Example~\ref{ex:circle}, and inspired us to significantly
%  expand and rewrite this paper. 
%}

\date{July 12, 2016}
\maketitle
\begin{abstract}
  We investigate how a C*-algebra could consist of functions on a noncommutative set:
  a \emph{discretization} of a C*-algebra $A$ is a $*$-homomorphism $A \to M$ that factors 
  through the canonical inclusion $C(X) \subseteq \linfty(X)$ when restricted to a commutative 
  C*-subalgebra.
  Any C*-algebra admits an injective but nonfunctorial discretization, as well as a possibly noninjective
  functorial discretization, where $M$ is a C*-algebra.
  Any subhomogenous C*-algebra admits an injective functorial discretization, where $M$ is a 
  W*-algebra.
  However, any functorial discretization, where $M$ is an AW*-algebra, must trivialize $A = B(H)$ for
  any infinite-dimensional Hilbert space $H$.
\end{abstract}

\section{Introduction}

In operator algebra it is common practice to regard C*-algebras as noncommutative 
analogues of topological spaces, and to regard W*-algebras as noncommutative
analogues of measurable spaces. 
What would it mean to make precise how a C*-algebra is a `noncommutative ring of
continuous functions'? 
Several natural approaches to this question cannot faithfully represent examples as simple as matrix algebras $\M_n(\C)$~\cite{reyes:obstructing, bergheunen:extending, reyes:sheaves, benzvimareyes:kochenspecker}.
Such obstructions suggest more carefully considering what `noncommutative sets' in the foundations of noncommutative geometry should be, before attempting to topologize them.

This article explores the idea of embedding the C*-algebra in an appropriate noncommutative algebra of `bounded functions on the noncommutative set underlying its spectrum', just like any topological space embeds in a discrete one.
More precisely, consider the case of a commutative C*-algebra $A$.
A representation of $A$ as operating on a Hilbert space $H$ is equivalent to a $*$-homomorphism $A \to B(H)$.
Similarly, representing $A$ as continuous complex-valued functions on a compact Hausdorff space $X$ can equivalently be viewed as a $*$-homomorphism $A \to \linfty(X)$ to the algebra of bounded functions on the set $X$. 
More generally, representating $A$ as (discrete) functions on a set $X$ can equivalently be viewed as a $*$-homomorphism to the algebra $\C^X$ of all functions on $X$.

In the spirit of noncommutative geometry, we thus seek a category $\Alg$ of 
$*$-algebras to play the role of the dual to the category of `noncommutative sets'. 
This category should contain the commutative algebras $\linfty(X)$ (or $\C^X$) as a full subcategory, dual to the category of sets.
In keeping with the programme of taking commutative subalgebras seriously~\cite{heunen:faces,
reyes:obstructing,bergheunen:extending,reyes:sheaves,vdbergheunen:colim, 
heunenlandsmanspitters:topos,hamhalter:ordered,heunenreyes:activelattice,hamhalter:dye},
we posit that a representation of a C*-algebra as an algebra of functions on a noncommutative set should be an algebra homomorphism $\phi \colon A \to M$ for some $M$ in $\Alg$, whose restriction to every commutative C*-subalgebra $C \simeq C(X)$ of $A$ factors through the natural inclusion 
$C(X) \subseteq \linfty(X)$ via a morphism $\linfty(X) \to M$ in $\Alg$. We call such a map $\phi$ a \emph{discretization} of $A$.
\[\begin{tikzpicture}[xscale=3,yscale=1]
  \node (tl) at (0,1) {$A$};
  \node (tr) at (1,1) {$M$};
  \node (bl) at (0,0) {$C(X)$};
  \node (br) at (1,0) {$\linfty(X)$};
  \draw[->] (tl) to node[above] {$\phi$} (tr);
  \draw[right hook->] (bl) to (br);
  \draw[->] (bl) to (tl);
  \draw[->,dashed] (br) to (tr);
\end{tikzpicture}\]
Section~\ref{sec:discretization} makes this definition precise, relative to a parameterizing category $\Alg$ that can then remain unspecified. This approach to terminology gives most flexibility in investigating the open problem of finding a suitable noncommutative extension of the functor $C(X) \mapsto \linfty(X)$ before us. 
We show that every C*-algebra admits a discretization into a C*-algebra $M$ that is injective but nonfunctorial. 
We also show that there is a universal candidate for a functorial discretization into the category of C*-algebras, but it remains open whether this functorial discretization is injective for every C*-algebra.

In Section~\ref{sec:profinite} we show that a sizeable class of C*-algebras that are 
`close to being commutative' does indeed have injective functorial discretizations, namely the \emph{subhomogeneous} algebras: subalgebras of $\M_n(C)$ for some commutative C*-algebra $C$. The discretization is achieved by profinite completion,
suggesting that profinite completion for subhomogeneous algebras is a noncommutative substitute for the `underlying set functor' that sends a compact Hausdorff space to its underlying discrete space.

On the other hand, in Section~\ref{sec:nogo} we show that no subcategory of W*-algebras, or even AW*-algebras, can be dual to noncommutative sets in the sense of injectively discretizing every C*-algebra.
In particular, every functor from C*-algebras to AW*-algebras taking each C*-algebra to a discretization must trivialize $A = B(H)$ for any infinite-dimensional Hilbert space $H$.
A number of related examples and obstructions are discussed, including separable algebras $A$ for which the same trivialization occurs.
Viewing $*$-homomorphisms out of a C*-algebra as representing it by functions on a noncommutative set dates back at least to Akemann~\cite{akemann:stoneweierstrass} and Giles and Kummer~\cite{gileskummer:topology}, who took the representation to be the canonical homomorphism $A \to A^{**}$ into the bidual.
They noted~\cite[p.~10]{akemann:gelfand} that their theory was not functorial. 
Our obstructions amplify this observation by suggesting that W*-algebras indeed cannot play the role of `noncommutative $\linfty(X)$-algebras' for C*-algebras as large as $B(H)$.

The article concludes with a discussion in Section~\ref{sec:conclusion} of the implications of our obstructions, with an eye toward future work on the problem of finding injective functorial discretizations of all C*-algebras.

\section{Discretization}
\label{sec:discretization}

We assume throughout this article that all rings, algebras, and subalgebras are unital, and that all homomorphisms preserve units. Write $\Spec(C)$ for the Gelfand spectrum of a commutative C*-algebra $C$.
Write $\Cstar$ for the category of C*-algebras with $*$-homomorphisms and $\Wstar$ for the subcategory of W*-algebras with normal $*$-homomorphisms.

Recall that a \emph{pro-C*-algebra}~\cite{phillips:inverse, phillips:applications} is a topological $*$-algebra that is a directed (or ``inverse'') limit in the category of topological $*$-algebras of a system of C*-algebras. Pro-C*-algebras with continuous $*$-homomorphisms form a category $\proCstar$. The algebra $\C^X$ of all complex-valued functions on a set $X$ equipped with its topology of pointwise convergence is a pro-C*-algebra, as it is the directed limit of the finite-dimensional C*-algebras $\C^S$ for all finite subsets $S \subseteq X$.%, and therefore is a pro-C*-algebra.

\begin{lemma}\label{lem:equivalence}
  The functors $X \mapsto \linfty(X)$ and $X \mapsto \C^X$ are contravariant equivalences between the category of sets and full subcategories of $\Wstar$ and $\proCstar$. 
\end{lemma}
\begin{proof}
  The proof for the functor $\linfty$ can be found in~\cite[Section~6.1]{weaver:quantization}. We sketch an argument that covers both functors.

  It is rather clear that each of the above assignments forms a contravariant functor into the specified
  category. It only remains to show that each is naturally bijective on Hom-sets. 
  Fix $x \in X$.
  Let $\ev_x \colon \C^X \to \C$ denote the continuous $*$-homomorphism given by evaluation at
  $x$, whose restriction to $\linfty(X)$ is normal.
  The maps $X \to \proCstar(\C^X,\C)$ and $X \to \Wstar(\linfty(X),\C)$, given in each case 
  by $x \mapsto \ev_x$, are both bijections; this follows by verifying that the kernel of either kind of
  morphism $\C^X \to \C$ or $\linfty(X) \to \C$ is generated as an ideal by a characteristic function 
  $\chi_S$, which entails that $S = X \setminus \{x\}$ for some $x \in X$.

  Now the argument that the functors in question are bijective on Hom-sets is purely formal, and
  can be proved by essentially the same argument as the one given in the algebraic context
  in~\cite[Theorem~4.7]{iovanovmesyanreyes:diagonal}.
\end{proof}

The previous lemma leads naturally to the following notion, in keeping with the programme 
of taking commutative subalgebras seriously. 
As mentioned in the introduction, the definition is made relative to a category $\Alg$ of complex algebras 
that is a candidate to contain `algebras of bounded functions on noncommutative sets.'

\begin{definition}\label{def:discretization}
  Let $\Alg$ denote a category of $\C$-algebras containing the algebras $\linfty(X)$ for any set $X$ with their normal *-homomorphisms. Given a C*-algebra $A$, a \emph{(bounded) $\Alg$-discretization} is a homomorphism $\phi \colon A \to M$ whose restriction to each commutative C*-subalgebra $C$ of $A$ factors through the natural inclusion $C \to \linfty(\Spec(C))$ via a morphism $\phi_C \colon \linfty(\Spec(C)) \dashrightarrow M$ in $\Alg$.
  \[\begin{tikzpicture}[xscale=3,yscale=1]
    \node (tl) at (0,1) {$A$};
    \node (tr) at (1,1) {$M$};
    \node (bl) at (0,0) {$C$};
    \node (br) at (1,0) {$\linfty(\Spec(C))$};
    \draw[->] (tl) to node[above] {$\phi$} (tr);
    \draw[right hook->] (bl) to (br);
    \draw[left hook->] (bl) to (tl);
    \draw[->,dashed] (br) to node[right] {$\phi_C$} (tr); 
  \end{tikzpicture}\]
  We call a discretization $\phi$ \emph{faithful} when it is injective and all $\phi_C$ can be chosen injective.
  We call $\phi$ \emph{compatible} if the morphisms $\phi_C$ can be chosen such that $\phi_C$ factors through $\phi_D$ via the induced morphism $\linfty(\Spec(C)) \to \linfty(\Spec(D))$ for commutative C*-subalgebras $C \subseteq D \subseteq A$.
\end{definition}

When $\Alg$ is $\Cstar$ or $\Wstar$ above, we will speak of C*- or W*-discreti\-zations instead of $\Alg$-discretizations. 

\begin{proposition}\label{prop:pushoutdiscretization}
  Every C*-algebra has a faithful C*-discretization.
\end{proposition}
\begin{proof}
  Write $L$ for the functor $C \mapsto \linfty(\Spec(C))$.
  Given a finite family $S = \{C_1, \ldots, C_n\}$ of commutative C*-subalgebras of $A$, write $A_S$ for the colimit in $\cat{Cstar}$ of the diagram whose objects are $A$, the $C_i$, and the $L(C_i)$, along with the inclusions of each $C_i$ into both $A$ and $L(C_i)$. This can be constructed up to isomorphism as an iterated amalgamated free product: 
  \[
    A_S \simeq ( \cdots ((A *_{C_1} L(C_1)) *_{C_2} L(C_2)) \cdots )*_{C_n} L(C_n).
  \]
  Thus the natural maps from $A$ and the $L(C_i)$ into $A_S$ are all embeddings; see~\cite[Theorem~3.1]{blackadar:pushout} or~\cite[Theorem~4.2]{pedersen:pullbacks}.

  The finite families $S$ of commutative C*-subalgebras of $A$ form a directed set under inclusion. Consider the directed colimit $M = \colim_S A_S$.  By construction the mediating map $\phi \colon A \to M$ is a C*-discretization. For finite subfamilies $S \subseteq T$ of commutative C*-subalgebras of $A$, the induced map $A_S \to A_T$ is injective because $A_T$ is formed from $A_S$ by iterated pushouts.
  Thus the natural maps $A_S \to M$ are injective~\cite[Theorem~1]{takeda:colimit}, from which it follows that $\phi$ is faithful.
\end{proof}

The discretization $\phi \colon A \to M$ constructed in the proof above is not compatible: for commutative C*-subalgebras $C \subsetneq D \subseteq A$, the algebra $M$ is obtained by gluing together distinct copies of $L(C)$ and $L(D)$ without regard to the 
natural inclusion $L(C) \to L(D)$. In Theorem~\ref{thm:universalfunctor} below we modify the construction to ensure compatibility, with the caveat that we no longer know that the discretization is even injective. 
This universally constructed C*-discretization will in fact satisfy the following natural condition.

\begin{definition}
  Let $\Alg$ be a category as in Definition~\ref{def:discretization}.
  A \emph{functorial $\Alg$-discretization} is a functor $F \colon \Cstar \to \Alg$ together with natural homomorphisms $\eta_A \colon A \to F(A)$ such that $\eta_C$ for each commutative C*-algebra $C$ turns into the natural inclusion $C \to \linfty(\Spec(C))$ by a natural isomorphism $F(C) \simeq \linfty(\Spec(C))$.
\end{definition}

A functorial discretization automatically gives compatible discretizations $A \to F(A)$ for every C*-algebra $A$:
writing $i_C \colon C \to A$ for the inclusion of a commutative C*-subalgebra gives the following commutative diagram.
\[\begin{tikzpicture}[xscale=3,yscale=1.5]
  \node (tl) at (0,1) {$A$};
  \node (tr) at (1,1) {$F(A)$};
  \node (bl) at (0,0) {$C$};
  \node (br) at (1,0) {$F(C)$ \rlap{$\simeq\linfty(\Spec(C))$}};
  \draw[->] (tl) to node[above] {$\eta_A$} (tr);
  \draw[right hook->] (bl) to node[above] {$\eta_C$} (br);
  \draw[left hook->] (bl) to node[left] {$i_C$} (tl);
  \draw[->,dashed] (br) to node[right] {$F(i_C)$} (tr); 
\end{tikzpicture}\]
Compatibility follows by applying $F$ to successive inclusions $C \subseteq D \subseteq A$.

Write $\cCstar$ for the full subcategory of $\Cstar$ of commutative C*-algebras.
Write $\cC(A)$ for the small subcategory of $\cCstar$ consisting of the commutative C*-subalgebras of a C*-algebra $A$ with their inclusion morphisms; we also view this as a partially ordered set.

\begin{theorem}\label{thm:universalfunctor}
  The functor $F \colon \Cstar \to \Cstar$ given by 
  \[
    F(A) = \colim\limits_{C \in \cC(A)} A *_C \linfty(\Spec(C))
  \]
  equipped with the naturally induced $*$-homomorphisms $\eta_A \colon A \to F(A)$
  is a functorial C*-discretization.
  For each C*-algebra $A$, the C*-discretization $A \to F(A)$ is universal among all compatible C*-discretizations of $A$. 
  Thus $F$ is universal among all functorial C*-discretizations.
\end{theorem}
\begin{proof}
  We follow the idea of~\cite[Theorem~2.15]{reyes:obstructing} but with arrows reversed.
  Write $L = \linfty \circ \Spec \colon \cCstar \to \Cstar$. 
  The assignment $A \mapsto \cC(A)$ is a functor to the category of small categories.
  Given a C*-algebra $A$, the assignment $C \mapsto A *_C L(C)$ is functorial $\cC(A) \to \Cstar$. 
  So $F(A) = \colim_{C \in \cC(A)} A *_C L(C)$ defines a functor $F \colon \Cstar \to \Cstar$. 
  Moreover, the induced $*$-homo\-morphisms $\eta_A \colon A \to F(A)$ are natural by construction. 
  Finally, if $A$ is commutative so that $A \in \cC(A)$, then  one naturally has an isomorphism 
  $F(A) \simeq \linfty(\Spec(A))$ that turns 
  $\eta_A$ into the inclusion $A \to \linfty(\Spec(A))$.
  Thus $F$ is a functorial C*-discretization.

  To verify universality of $\eta_A$, fix a compatible C*-discretization
  $\phi \colon A \to M$. Each $C \in \cC(A)$ then makes the following outer square commute.
  \[\begin{tikzpicture}[xscale=4,yscale=1.5]
    \node (tl) at (0,1) {$A$};
    \node (tr) at (1,1) {$M$};
    \node (bl) at (0,0) {$C$};
    \node (br) at (1,0) {$L(C)$};
    \node (m) at (.5,.5) {$A *_C L(C)$};
    \draw[->] (tl) to node[above] {$\phi$} (tr);
    \draw[->] (bl) to (br);
    \draw[->] (bl) to (tl);
    \draw[->] (br) to node[right] {$\phi_C$} (tr); 
    \draw[->] (tl) to (m);
    \draw[->] (br) to (m);
    \draw[->,dashed] (m) to (tr);
  \end{tikzpicture}\]
  The morphisms $\phi$ and $\phi_C$ factor uniquely through the pushout $A *_C L(C)$.
  Compatibility of the $\phi_C$ means that these uniquely determined morphisms form a 
  cocone from the diagram of the $A *_C L(C)$ to $M$. Thus we obtain a $*$-homomorphism 
  $F(A) = \colim_{C \in \cC(A)} A *_C L(C) \to M$ through which $\phi$ factors uniquely,
  as desired.

  Finally, if $(F',\eta')$ is any functorial C*-discretization, then by the local universality of the previous paragraph the natural morphisms $\eta'_A \colon A \to F'(A)$ factor uniquely through $\eta_A \colon A \to F(A)$, from which it readily follows that $F'$ factors through a unique natural transformation $F \Rightarrow F'$ whose composite with $\eta$ is $\eta'$.
\end{proof}

Whereas the `incompatible' discretization of Proposition~\ref{prop:pushoutdiscretization} is faithful, it is not clear whether the natural C*-discretizations $A \to F(A)$ of the last theorem are faithful or even injective. Abstract nonsense alone does not answer this question.

\begin{question}\label{question}
  Is the universal functorial C*-discretization $\eta_A \colon A \to F(A)$ of Theorem~\ref{thm:universalfunctor} injective or faithful for every C*-algebra $A$? Equivalently, does every C*-algebra have an injective or faithful compatible C*-discretization?
\end{question}

\begin{remark}
  The definitions and results above carefully used the Gelfand spectrum $\Spec(C)$ of a commutative C*-algebra $C$. 
  Henceforth we loosen notation, and write $C = C(X)$ for an arbitrary commutative C*-algebra, and $C \simeq C(X)$ for an arbitrary commutative C*-subalgebra of a C*-algebra $A$.
\end{remark}

Recall from Lemma~\ref{lem:equivalence} that sets may also be encoded algebraically through the algebra of discrete (possibly unbounded) functions as $X \mapsto \C^X$. The rest of the paper will also discuss `unbounded' discretizations.

\begin{definition}\label{def:unboundeddiscretization}
  Let $\Alg$ denote a category of $\C$-algebras containing the algebras $\C^X$ for any set $X$ with the $*$-homomorphisms that are continuous with respect to the topology of pointwise convergence. Given a C*-algebra $A$, an \emph{unbounded $\Alg$-discretization} is a homomorphism $\phi \colon A \to M$ whose restriction to each commutative C*-subalgebra $C \simeq C(X)$ of $A$ factors through the inclusion 
  $C(X) \to \C^X$ via a morphism $\phi_C \colon \C^X \dashrightarrow M$ in $\Alg$.
  \[\begin{tikzpicture}[xscale=3,yscale=1]
    \node (tl) at (0,1) {$A$};
    \node (tr) at (1,1) {$M$};
    \node at (-.3,0) {$C \simeq $};
    \node (bl) at (0,0) {$C(X)$};
    \node (br) at (1,0) {$\C^X$};
    \draw[->] (tl) to node[above] {$\phi$} (tr);
    \draw[right hook->] (bl) to (br);
    \draw[left hook->] (bl) to (tl);
    \draw[->,dashed] (br) to node[right] {$\phi_C$} (tr); 
  \end{tikzpicture}\]
  Define \emph{injective}, \emph{faithful}, and \emph{functorial} unbounded discretizations analogous to the bounded case.
  For $\Alg = \proCstar$ we refer to \emph{unbounded pro-C*-discretizations}.
\end{definition}

\section{Functorial discretizations through profinite completion}
\label{sec:profinite}

For a compact Hausdorff space $X$, the natural inclusion $C(X) \to \linfty(X)$ is a W*-discretization
of the corresponding commutative C*-algebra. 
Also, if $A$ is a finite-dimensional C*-algebra, then the identity map $A \to A$ is a
W*-discretization. 
This section provides a common generalization of these two examples: Theorems~\ref{thm:profinite} and~\ref{thm:subhomogeneous} below show that the profinite completion of a C*-algebra is a functorial discretization that is faithful for a large class of algebras.

For a C*-algebra $A$, let $\F(A)$ denote the family of closed ideals $I$ of $A$ for which $A/I$ is
finite-dimensional. Then $\F(A)$ is closed under finite intersections, as is readily verified by embedding $A/(I \cap J) \to A/I \oplus A/J$ for ideals $I,J \in \F(A)$. 
Thus the finite-dimensional C*-algebras $A/I$ for $I \in \F(A)$ form an inversely directed system.
We may take the directed limit of this system either in the category $\Cstar$ to obtain a C*-algebra,
or in the category of topological algebras to obtain a pro-C*-algebra. We denote these two
directed limits by
\begin{align*}
  P_b(A) &= \lim\nolimits_{I \in \F(A)} A/I \quad \mbox{computed in } \Cstar, \\
  P_u(A) &= \lim\nolimits_{I \in \F(A)} A/I \quad \mbox{computed in } \proCstar.
\end{align*}

Given a $*$-homomorphism $f \colon A \to B$ and $J \in \F(B)$, the induced embedding
$A/f^{-1}(J) \hookrightarrow B/J$ ensures that $f^{-1}(J) \in \F(A)$. Universality provides a composite $*$-homomorphism
\[
  P_b(A) = \lim\nolimits_{I \in \F(A)} A/I 
  \;\to\; \lim\nolimits_{J \in \F(B)} A/f^{-1}(J) 
  \;\to\; \lim\nolimits_{J \in \F(B)} B/J = P_b(B)
\]
making the assignments $P_b$ and $P_u$ functorial. 

Notice that the diagram over which the limit $P_b(A)$ is computed consists of W*-algebras with
normal $*$-homomorphisms. 
The subcategory $\Wstar$ of $\Cstar$ is closed under limits since the forgetful functor
$\Wstar \to \Cstar$ is  right adjoint to the universal enveloping W*-algebra functor~\cite{dauns:tensorproduct}.
Thus $P_b(A)$ is a W*-algebra, and for $f \colon A \to B$ in $\Cstar$ the induced morphism $P_b(f) \colon P_b(A) \to P_b(B)$ is a normal $*$-homomorphism.
Thus $P_b$ is a functor $\Cstar \to \Wstar$.

Each of the two functors $P_b$ and $P_u$ is a kind of profinite completion~\cite{elhartiphillipspinto:profinite}.

\begin{definition}
  We call $P_b \colon \Cstar \to \Wstar$ the \emph{bounded profinite completion}, 
  and $P_u \colon \Cstar \to \proCstar$ the \emph{unbounded profinite completion}.
\end{definition}

Let $b(P) \subseteq P$ denote the set of \emph{bounded} elements of a pro-C*-algebra $P$: those elements whose spectrum forms a bounded subset of $\C$. This is a C*-algebra that lies densely in $P$~\cite[Proposition~1.11]{phillips:inverse}.

\begin{proposition}
  If $A$ is a C*-algebra, then $P_b(A) \simeq b(P_u(A))$: the W*-algebra $P_b(A)$ is $*$-isomorphic to the algebra of bounded elements of the pro-C*-algebra $P_u(A)$.
\end{proposition}
\begin{proof}
  Suppose that a C*-algebra $B$ forms a cone over the diagram of finite-dimensional algebras $A/I$
  for $I \in \F(A)$. Then $B$ also forms a cone over this diagram in the category $\proCstar$, and this cone factors uniquely through a morphism $B \to P_u(A)$. But the image of this morphism lands in the C*-algebra $b(P_u(A))$
  ~\cite[Corollary~1.13]{phillips:inverse}. Thus $b(P_u(A))$ satisfies the universal property of $\lim_{I \in \F(A)} A/I$ computed in $\Cstar$.
  It follows that the map $P_b(A) \to P_u(A)$ induced by the universal property of the latter is an
  isomorphism onto $b(P_u(A))$. 
\end{proof}

Henceforth we identify $P_b(A)$ with the dense subalgebra $b(P_u(A)) \subseteq P_u(A)$.
Invoking the universal property of limits once again, for each C*-algebra $A$ there
is a $*$-homomorphism $\eta_A \colon A \to P_b(A) \subseteq P_u(A)$ that is natural in $A$. 
This map makes $P_b$ and $P_u$ into functorial discretizations.

\begin{theorem}\label{thm:profinite}
  Bounded profinite completion is a functorial W*-discretization.
  Unbounded profinite completion is an unbounded functorial pro-C*-discretization.
\end{theorem}
\begin{proof}
  For a commutative C*-algebra $C = C(X)$, each $I \in \F(C)$ is of the form $I = I_S =
  \{f \in C \mid f(S) = 0\}$ for some finite subset $S \subseteq X$. The surjection $C \twoheadrightarrow C/I \simeq C(S)$
  is Gelfand dual to the inclusion $S \hookrightarrow X$. Thus
  \begin{align*}
  P_b(C(X)) &= \lim\nolimits_{S \subseteq X} C(S) \simeq \linfty(X), \\
  P_u(C(X)) &= \lim\nolimits_{S \subseteq X} C(S) \simeq \C^X,
  \end{align*}
  and under these isomorphisms the natural map $\eta_C \colon C \to P_b(C) \subseteq P_u(C)$
  corresponds to the natural inclusion $C(X) \hookrightarrow \linfty(X) \subseteq \C^X$.

  It remains to verify that these functors behave as expected on morphisms. Fix a $*$-homomorphism
  $f \colon B = C(Y) \to C = C(X)$, which is Gelfand dual to a continuous function $\fhat \colon X \to Y$.
  For any finite set $S \subseteq X$, the restriction of $\fhat$ to $S \to \fhat(S)$
  is Gelfand dual to $C(\fhat(S)) \simeq B/f^{-1}(I_S) \to
  C/I_S \simeq C(S)$. Taking the directed limit in $\Wstar$ over finite subsets $S \subseteq X$, we see that the induced map 
  $P_b(f) \colon P_b(B) \to P_b(C)$ corresponds to $\linfty(\fhat)$ under the isomorphisms $P_b(B) \simeq \linfty(Y)$
  and $P_b(C) \simeq \linfty(X)$. This completes the proof for $P_b$; the 
  analogous argument in $\proCstar$ also holds for $P_u$. 
\end{proof}

\begin{example}
  Let $A=\M_n(C(X))$ for a compact Hausdorff space $X$. Then $P_b(A) = \M_n(\linfty(X))$ and $P_u(A) = \M_n(\C^X)$.
\end{example}
\begin{proof}
  Write $C = C(X)$, and recall that every closed ideal $J \subseteq \M_n(C)$ is of the form $\M_n(I)$ for some closed ideal $I\subseteq C$~\cite[Corollary~17.8]{lam:lectures}. Such an ideal $J$ has finite codimension in $A$ if and only if $I$ has finite codimension in $C$. Thus 
  \begin{align*}
    P_b(A) &=\; \lim\nolimits_{J \in \F(A)} A/J \;=\; \lim\nolimits_{I \in \F(C)} \M_n(C)/\M_n(I) \\
    &\simeq\; \lim\nolimits_{I \in \F(C)} \M_n(C/I) \;\simeq\; \M_n(\linfty(X))
  \end{align*}
  and similarly $P_u(A) \simeq \M_n(\C^X)$.
\end{proof}

Let us emphasize that, even though the profinite completion functors yield discretizations of \emph{all} C*-algebras, there are many C*-algebras $A$ for which $P_b(A) = P_u(A) = 0$ is trivial. Indeed, if $A$ is any C*-algebra with no finite-dimensional
representations, then by construction of the profinite completions we necessarily have 
$P_b(A) = P_u(A) = 0$. Example include: the algebra $B(H)$ of bounded operators on an infinite-dimensional Hilbert space $H$; the CCR algebra~\cite{petz:ccr}; the Calkin algebra $B(H)/K(H)$; and the (separable) Cuntz algebra $\mathcal{O}_n$ generated by $n \geq 2$ isometries~\cite{cuntz:simple}.
Thus it is interesting to see which algebras have injective or faithful discretizations to their profinite completion. This is addressed in the next theorem.

Recall that a C*-algebra $A$ is \emph{residually finite-dimensional} when it has a faithful family of finite-dimensional representations. Similarly, $A$ is \emph{subhomogeneous} when there is an integer $n \geq 1$ such that every irreducible representation of $A$ has dimension at most~$n$; this is equivalent~\cite[Proposition~IV.1.4.3]{blackadar:algebras} to $A$ being isomorphic to a C*-subalgebra of $\M_k(C)$ for a commutative C*-algebra $C$ and an integer $k \geq 1$. For a point $x$ in a set $X$, we let $\delta_x = \chi_{\{x\}} \in \linfty(X) \subseteq \C^X$ denote the indicator function of the singleton $\{x\}$.

\begin{theorem}\label{thm:subhomogeneous}
  For a C*-algebra $A$, the functorial discretizations $P_b$ and $P_u$ are:
  \begin{enumerate}
    \item injective if and only if $A$ is residually finite-dimensional;
    \item faithful if $A$ is subhomogeneous.
  \end{enumerate}
\end{theorem}
\begin{proof}
  (i) If $A$ is residually finite-dimensional, every nonzero $a \in A$ allows $I_a \in \F(A)$ with $a \notin I_a$ (meaning that $a$ has nonzero image in $A/I_a$). Thus $a$ is not in the kernel of $\eta_A \colon A \to \lim_{I \in \F(A)} A/I = P_b(A) \subseteq P_u(A)$. Hence $\eta_A$ is injective. (See also~\cite[Lemma~1.10]{elhartiphillipspinto:profinite}.)
  The converse follows directly from the definition.

  (ii) Consider a commutative C*-subalgebra $C(X) \subseteq A$, and $x \in X$. Because the homomorphisms $\linfty(X) \simeq P_b(C(X)) \to P_b(A)$ and $\C^X \simeq P_u(C(X)) \to P_u(A)$ are respectively normal and continuous, it suffices to show
  that $\delta_x \in \linfty(X) \subseteq \C^X$ is not in their kernel.
  Indeed, the kernel $I$ of either morphism is an ideal generated by a characteristic function $\chi_S$ for some $S \subseteq X$, so that $I$ contains exactly those $\delta_x$ with $x \in S$. Hence if all $\delta_x \notin I$, then $S = \varnothing$ and therefore $I = 0$.

  Evaluation at $x$ is a pure state on $C(X)$, which extends~\cite[II.6.3.2]{blackadar:algebras} to a pure state $\rho_x$ on $A$. Because $A$ is subhomogeneous, the GNS construction applied to $\rho_x$ yields a finite-dimensional representation $\pi \colon A \to B(\C^n) \simeq \M_n(\C)$ for some integer 
  $n \geq 1$, with cyclic vector $v_x \in \C^n$. Let $I \in \F(A)$ denote the kernel of $\pi$. 
  The induced $*$-homomorphism $\psi \colon \linfty(X) \to A/I \hookrightarrow \M_n(\C)$ has image 
  isomorphic to $C(S)$ for some finite subset $S \subseteq X$; in fact, this set $S$ is characterized
  as those pure states on $C(X)$ that are induced by vector states of the representation $\pi$.
  Now $\pi(f) v_x = f(x) v_x$ for $f \in C(X)$ by construction of $\pi$. 
  Thus $x \in S$, so that $\delta_x$ is not in the kernel of $\psi$.
  It follows that $\delta_x$ has nonzero image in each of the limit algebras $P_b(A)$ and $P_u(A)$,
  as desired.
\end{proof}

\begin{remark}
  For C*-algebras $A$ that are residually finite-dimensional but not subhomogeneous, the natural map $A \to P_b(A)$ is technically an injective discretization, but it does not satisfy all desiderata for an `algebra of bounded functions on the noncommutative underlying set' of $A$. 
  Consider the C*-sum $A = \overline{\bigoplus_{k=1}^\infty} \M_k(\C)$. 
  Let $I_n\subseteq A$ denote the kernel of the projection $A \twoheadrightarrow \M_1(\C) \oplus \cdots \oplus \M_n(\C)$ onto the first $n$ components.
  By an argument similar to that in~\cite[Lemma~7.5]{kaplansky:structure}, the kernel 
  of any finite-dimensional representation of $A$ must contain some $I_n$.
  It follows that the $I_n$ form a cofinal chain in $\F(A)$, so that the profinite completion
  \[
    A \to P_b(A) \simeq \lim\nolimits_{n \to \infty} A/I_n \simeq A
  \]
  is an isomorphism. But this is far from the behavior
  one would expect when comparing to the commutative example 
  $C = \overline{\bigoplus_{k=1}^\infty} \C \simeq \linfty(\N) \simeq C(\beta \N)$; the 
  profinite completion $C \to P_b(C)$ corresponds under this isomorphism to the embedding
  $C \simeq C(\beta \N) \to \linfty(\beta\N)$, indicating that $C$ is `far below' $P_b(C)$
  as a subalgebra.
\end{remark}

Almost all faithful discretizations of C*-algebras we know are supplied by Theorem~\ref{thm:subhomogeneous} above. We conclude this section by describing another significant example of a faithful compatible discretization that is not of this form.

\begin{example}
  For an infinite-dimensional Hilbert space $H$, consider the \mbox{C*-subalgebra} $A = \C \oplus K(H)$ of $B(H)$ generated by the identity and the compact operators. The embedding $A \hookrightarrow B(H)$ is a faithful compatible W*-discretization.
\end{example}
\begin{proof}
  Any commuting set of self-adjoint compact operators on $H$ has an orthonormal basis of $H$ of simultaneous eigenvectors, so the same remains true for commuting sets of self-adjoint operators in $A$.
  Let $C \simeq C(X) \subseteq A$ be a commutative C*-subalgebra. For $x \in X$ let $p_x\in B(H)$ denote
  the projection onto the simultaneous eigenspace  $\{v \in H \mid f \cdot v = f(x)v \mbox{ for all } 
  f \in C\}$. Now each $p_x \neq 0$ and $\sum p_x = 1$ in $B(H)$. It follows that the W*-subalgebra $W_C$ generated by the $p_x$ is
  isomorphic to $\linfty(X)$, and the fact that $fp_x = p_xf = f(x) \cdot p_x$ for all $f \in C$
  guarantees that the natural inclusion $C \subseteq W_C$ corresponds under this isomorphism
  to the natural inclusion $C(X) \subseteq \linfty(X)$. Thus the discretization is faithful.

  Compatibility for commutative C*-subalgebras $C \subseteq D \subseteq A$ is readily established 
  from the simple observation that a simultaneous eigenspace for $D$ restricts to a simultaneous
  eigenspace for $C$.
\end{proof}

The example above is a faithful compatible W*-discretization for which we do not know of any extension 
to an unbounded discretization.

\section{Obstructions to discretizations with many projections}
\label{sec:nogo}

Can the bounded faithful functorial W*-discretization for subhomogeneous C*-algebras of 
Theorem~\ref{thm:subhomogeneous} be extended to general C*-algebras through some method
other than profinite completion? Perhaps surprisingly, we prove in this section that the answer is no: 
any W*-discretization of the algebra $B(H)$ for an infinite-dimensional Hilbert 
space $H$ is necessarily zero. 
In fact, the obstruction is even more serious: if we replace the category of W*-algebras
(`noncommutative measurable spaces') with the category of AW*-algebras~\cite{kaplansky:projections, berberian} (`noncommutative complete Boolean algebras'~\cite{heunenreyes:activelattice}), the obstruction persists. 

The next definition is crucial to our obstructions, and relies on the following notions from measure theory.
An \emph{atom} of a measure space $(X,\mu)$ is a measurable subset $U \subseteq X$ with $\mu(U)>0$, 
such that $\mu(V)<\mu(U)$ implies $\mu(V) = 0$ for any measurable subset $V \subseteq U$. 
An atom of a regular Borel measure on a locally compact Hausdorff space is necessarily a 
singleton~\cite[2.IV]{knowles:nonatomic}. 
A measure is \emph{diffuse} if it has no atoms.
We will say that a positive linear functional $\psi \colon C(X) \to \C$ of a commutative 
C*-algebra, given by $\psi(f)=\int f \, d\mu$ for a regular Borel measure $\mu$ on $X$, 
is \emph{diffuse} when $\mu$ is diffuse.

\begin{definition}\label{def:relativelydiffuse}
  Let $A$ be a C*-algebra. A pair of commutative C*-sub\-algebras $C$ and $D$ is \emph{relatively diffuse} when every extension of a pure state of $D$ to a state of $A$ restricts to a diffuse state on $C$.
\end{definition}

\begin{example}\label{ex:kadisonsinger}
  Consider the separable Hilbert space $H=L^2[0,1]$, and the C*-algebra $A=B(H)$.
  Write $D$ for the discrete maximal abelian W*-subalgebra generated by the projections $q_n$ 
  onto the Fourier basis vectors $e_n = \exp(2\pi i n -)$ for $n \in \Z$, and $C$ for the continuous
  maximal abelian W*-subalgebra $\Linfty[0,1]$.
  Then $C$ and $D$ are relatively diffuse.
\end{example}
\begin{proof}
  There is a canonical conditional expectation $E \colon A \to D$ that sends $f \in A$ to its diagonal part $\sum q_n f q_n$.
  For $f \in C$ then $E(f) = \int_0^1 f(t)\, \mathrm{d}t$ because
  \[
    \inprod{f e_n}{e_n} 
    = \int_0^1 f(t) \cdot e^{2 \pi int} \cdot \overline{e^{2 \pi int}} \, \mathrm{d}t 
    = \int_0^1 f(t) \, \mathrm{d}t.
  \]
  Because $\psi$ is a pure state of $D$ now $\psi = \psi \circ E$ by the solution of the Kadison-Singer problem~\cite{marcusspielmansrivastava:kadisonsinger}.
  Hence $\psi(f) = \psi(E(f)) = \psi( \int_0^1 f(t)\,\mathrm{d}t) = \int_0^1 f(t)\,\mathrm{d}t$.
\end{proof}

\begin{example}\label{ex:separable}
  For $H = L^2[0,1]$, consider any separable C*-subalgebra $C \subseteq \Linfty[0,1] \subseteq B(H)$ 
  for which the state $f \mapsto \int_0^1 f(t) \, dt$ is diffuse (such as $C = C[0,1]$). 
  Then there is a separable C*-subalgebra $A \subseteq B(H)$ containing $C$ and a commutative 
  C*-subalgebra $D$ generated by projections, with $C$ and $D$ relatively diffuse.
\end{example}
\begin{proof}
  Let $e_n$ and $E$ be as in Example~\ref{ex:kadisonsinger}.
  Because $C$ is separable, we can fix a sequence $\{f_i\}_{i=1}^\infty$ of elements whose linear span is dense
  in $C$. For each $f_i$ and for each integer $j \geq 1$, the positive solution to the 
  paving conjecture~\cite{marcusspielmansrivastava:kadisonsinger} ensures that there is a finite set of 
  projections $p_k = \smash{p_k^{(i,j)}}$ in the discrete maximal abelian subalgebra of $B(H)$ relative 
  to the Fourier basis $e_n$ with $\sum p_k = 1$ and $\| p_k (f_i - E(f_i))p_k \| \leq 1/j$.
  Let $D$ be the commutative C*-subalgebra of $B(H)$ generated by the $\smash{p^{(i,j)}_k}$  for all $i$, $j$, 
  and $k$.
  Let $A$ be the C*-subalgebra of $B(H)$ generated by $C$ and $D$; as both $C$ and $D$ are countably
  generated, the same is true of $A$, whence $A$ is separable.
  An argument familiar in the literature on the Kadison-Singer problem (as in~\cite[p310]{anderson:extensions}) 
  shows that any extension of a pure state $\psi_0$ on $D$ to a state $\psi$ on $A$ satisfies 
  $\psi(f) = \psi_0(E(f))$ for all $f \in C$. The same computation as in Example~\ref{ex:kadisonsinger}
  shows that $\psi(f) = \int_0^1 f(t) \, dt$, which is diffuse on $C$ by hypothesis.
\end{proof}

\begin{remark}
  It is possible to modify Examples~\ref{ex:kadisonsinger} and~\ref{ex:separable} so that the conclusions
  can be reached without using the full force of Kadison-Singer. In either case, identify the algebra 
  $C = C(\mathbb{T})$ of continuous functions on the unit circle with the subalgebra 
  $\{f \mid f(0) = f(1)\} \subseteq C[0,1] \subseteq B(H)$. 
  The algebra of Fourier polynomials---or more generally, the Wiener algebra $A(\mathbb{T})$---is
  a dense subalgebra of $C$ and lies in the algebra $M_0 \subseteq B(H)$ of operators that
  are $l_1$-bounded in the sense of Tanbay~\cite{tanbay:extensions} with respect to the
  Fourier basis $\{e_n \mid n \in \Z\}$. Thus $C$ lies in the norm closure $M$ of $M_0$, and
  it was shown in~\cite{tanbay:extensions} (without the full force of Kadison-Singer)
  that every element of $M$ is compressible (that is, the operator $f - E(f)$
  satisfies paving with respect to the basis $e_n$ for any $f \in M$). The computations in either example
  given above may now proceed in the same manner.
\end{remark}

The relatively diffuse subalgebras $C$ and $D$ in the examples above had pure states 
of $D$ inducing a \emph{unique} diffuse state on $C$. We thank the referee for the 
following example which allows for possibly non-unique extensions.

\begin{example}\label{ex:circle}
  Let $A$ and $D$ be as in Example~\ref{ex:kadisonsinger}, but consider the commutative C*-subalgebra of $A$ generated by the bilateral shift $e_n \mapsto e_{n+1}$, and let $C$ be its bicommutant.
  Then $C$ and $D$ are relatively diffuse.
\end{example}
\begin{proof}
  Write $C_0$ for the C*-subalgebra generated by the shift $u \colon H \to H$; its Gelfand spectrum is the unit circle $\T = \{ \lambda \in \C \mid |\lambda|=1\}$~\cite[Problem~84]{halmos:problembook}. Let $f_n \in C(\T)$ be a decreasing sequence converging to the characteristic function $\delta_\lambda = \chi_{\{\lambda\}}$ of some $\lambda \in \T$. 
  Then, since the bounded sequence $(f_n)$ converges pointwise to $\delta_\lambda$, the sequence $(f_n(u))$  in $C_0$ converges strongly to the projection $\delta_\lambda(u)$ in $C$. But $\mathrm{lim}_n \inprod{f_n(u)(e_0)}{e_0} = \inprod{\delta_\lambda(e_0)}{e_0}$ vanishes because $u$ has no eigenvectors. 
  Hence $\|E(f_n(u))\| = \|\inprod{f_n(u)(e_0)}{e_0} 1_H\| \to 0$.
  Thus a state $\psi$ of $A$ that is pure on $D$ satisfies $\psi(f_n) = \psi(E(f_n)) \to 0$, and is therefore diffuse on $C$.
\end{proof}

Relatively diffuse pairs of commutative C*-subalgebras are inherited along $*$-homomorphisms, as follows.

\begin{lemma}\label{lem:functorialpair}
  Let $\phi \colon A \to B$ be a morphism in $\Cstar$.
  If two commutative C*-subalgebras $C, D \subseteq A$ are relatively diffuse,
  then so are $\phi(C),\phi(D) \subseteq B$.
\end{lemma}
\begin{proof}
  Fix a pure state $\psi_0$ on $\phi(D)$, and let $\psi$ be any extension to a state on $B$. 
  Then $\psi \circ \phi$ is a state on $A$ that extends $\psi_0 \circ \phi$ from $D$;
  observe that the latter is a pure state on $D$ as it is a composition of a $*$-homomorphism
  with a pure state.
  By hypothesis, the restriction of $\psi \circ \phi$ to $C$ is diffuse.
  As the restriction of $\phi$ to $C \twoheadrightarrow \phi(C)$ is Gelfand dual to the inclusion $\Spec(\phi(C)) \hookrightarrow \Spec(C)$ of a closed subspace,
  the measure on $\Spec(\phi(C))$ corresponding to $\psi|_{\phi(C)}$ is the restriction of the
  measure on $\Spec(C)$ corresponding to $\psi_0|_C$, which is diffuse.
  It follows that the restriction of $\psi$ to $C'$ is diffuse.
\end{proof}

The major result below and its many corollaries will refer to commutative diagrams of the following
kind, where $A$ is a C*-algebra with relatively diffuse commutative C*-subalgebras $C \simeq C(X)$ and 
$D \simeq C(Y)$.
\begin{equation}\label{eq:nogodiagram}
  \begin{aligned} 
    \begin{tikzpicture}[xscale=2.5,yscale=1.1]
         \node (C) at (-.35,2) {$C \simeq$}; 
         \node (D) at (-.35,0) {$D \simeq$};
	  \node (CX) at (0,2) {$C(X)$};
	  \node (A) at (0,1) {$A$};
	  \node (CY) at (0,0) {$C(Y)$};
	  \node (lX) at (1,2) {$\linfty(X)$};
	  \node (B) at (1,1) {$M$};
	  \node (lY) at (1,0) {$\linfty(Y)$};
	  \draw[right hook->] (CX) to (A);
	  \draw[left hook->] (CY) to (A);
	  \draw[->] (A) to node[above] {$\phi$} (B);
	  \draw[right hook->] (CX) to (lX);
	  \draw[right hook->] (CY) to (lY);
	  \draw[->] (lX) to node[right] {$\phi_C$} (B);
	  \draw[->] (lY) to node[right] {$\phi_D$} (B);
    \end{tikzpicture}
  \end{aligned}
\end{equation}

\begin{theorem}\label{thm:nogo}
  If a C*-algebra $A$ has relatively diffuse commutative C*-sub\-alge\-bras $C \simeq C(X)$ and $D \simeq C(Y)$, 
  and if there is a C*-algebra $M$ with $*$-homomorphisms $\phi$, $\phi_C$ and $\phi_D$ 
  making the diagram~\eqref{eq:nogodiagram} commute, then  for any $x \in X$ and $y \in Y$:
  \[
  \phi_C(\delta_x) \phi_D(\delta_y)=0.
  \] 
\end{theorem}
\begin{proof}
  Let $x \in X$ and $y \in Y$, and write $p=\phi_C(\delta_x)$ and $q=\phi_D(\delta_y)$. 
  Fix any state $\sigma$ on the C*-algebra $qBq$, and let $\psi$ 
  denote the state on $A$ given by $\psi(a) = \sigma(q\phi(a)q)$. 
  For $g \in D$, observe $\psi(g) = \sigma(\phi_D(\delta_y g \delta_y)) = \sigma(\phi_D(g(y) \delta_y)) = g(y) \sigma(q) = g(y)$,
  so that $\psi$ restricts to a pure state on $D$. 
  By hypothesis, the restriction of $\psi$ to $C$ is of the form $f \mapsto \int f \,d\mu$ 
  for some diffuse Radon measure $\mu$ on $X$. 
  Thus for each integer $n \geq 1$ we may find an open neighborhood $U_n$ of $x$ 
  with $\mu(U_n) \leq \frac{1}{n}$.
  Urysohn's lemma provides a continuous function $f_n \colon X \to [0,1]$ that vanishes 
  on $X \setminus U_n$ and satisfies $f_n(x)=1$.
  Since $\delta_x \leq f_n$ in $\linfty(X)$ we have $p = \phi_C(\delta_x) \leq \phi_C(f_n)$. 
  Positivity of $b\mapsto \sigma(qbq)$ yields
  \[
    \sigma(qpq) \leq \sigma(q\phi_C(f_n)q) = \psi(f_n) = \int f_n \, d\mu \leq \mu(U_n) \leq \frac{1}{n}.
  \]
  As $n \to \infty$ we find that $\sigma(pqp) = 0$ for all states $\sigma$ on $B$, making $q p q = 0$.
  It follows that $\|qp\|^2=\|qpq\|=0$ and thus $pq=(qp)^*=0$.
\end{proof}

Write $\AWstar$ for the category of AW*-algebras with $*$-homomorphisms whose restriction to the projection lattices preserve arbitrary least upper bounds\footnote{See~\cite[Lemma~2.2]{heunenreyes:activelattice} for further characterizations 
of these  morphisms.}; $\Wstar$ is a full subcategory.
We call $\AWstar$-discretizations \emph{AW*-discretizations}.

\begin{corollary}\label{cor:awstarnogo}
  If a C*-algebra $A$ has two relatively diffuse commutative C*-subalgebras,
  then any AW*-discretization $\phi \colon A \to M$ satisfies $M=0$. Consequently, every functorial AW*-discretization
  $F \colon \Cstar \to \AWstar$ has $F(A) = 0$ for such $A$.
\end{corollary}
\begin{proof}
  Let $C \simeq C(X)$ and $D \simeq C(Y)$ be the relatively diffuse commutative C*-subalgebras, 
  and let $\phi_C \colon \linfty(X) \to M$ and $\phi_D \colon \linfty(Y) \to M$ be the discretizing morphisms as in 
  Definition~\ref{def:discretization}, yielding a commuting diagram~\eqref{eq:nogodiagram}.  
  For $x \in X$ and $y \in Y$, set $p_x = \phi_C(\delta_x)$ and $q_y = \phi_D(\delta_y)$. 
  As $\sum \delta_x = 1_C$ and $\sum \delta_y = 1_D$ (in the sense of least upper bounds of orthogonal projections),
  and as $\phi_C$ and $\phi_D$ are morphisms in $\AWstar$, we have $\sum p_x = 1 = \sum q_y$ in $M$.
  By Theorem~\ref{thm:nogo}, each $p_x$ is orthogonal to all of the $q_y$, so that $p_x$ is orthogonal to $\sum q_y = 1 \in M$.
  Therefore $p_x=0$ for all $x \in X$, whence $1=\sum p_x = 0$ in $M$ and $M=0$.
\end{proof}

\begin{example}\label{ex:calkin}
  If there is a morphism $B(H) \to A$ in $\Cstar$ for some infinite-dimensional Hilbert space, then 
  $A$ has no nontrivial AW*-discretization.
\end{example}
\begin{proof}
  First note that $H$ as above is unitarily isomorphic to $L^2[0,1] \otimes H$, so $a \mapsto a \otimes 1$ is a
  $*$-homomorphism $B(L^2[0,1]) \to B(L^2[0,1]) \otimes B(H) \simeq B(H)$.
  Example~\ref{ex:kadisonsinger} along with Lemma~\ref{lem:functorialpair} show that $A$ contains
  a relatively diffuse commutative C*-subalgebras, so that Corollary~\ref{cor:awstarnogo} applies.
\end{proof}

In particular, by the last example the Calkin algebra $A=B(H)/K(H)$ has no nontrivial AW*-discretization for $H=L^2[0,1]$.

Theorem~\ref{thm:nogo} has the following consequence for purely ring-theoretic discretizations, with much tamer conclusion than those of Corollaries~\ref{cor:awstarnogo} or~\ref{cor:topologicalnogo}.

\begin{corollary}\label{cor:ringnogo}
  If a C*-algebra $A$ has relatively diffuse C*-subalgebras $C \simeq C(X)$ and $D \simeq C(Y)$, 
  and if there is a commutative diagram of the form~\eqref{eq:nogodiagram} where $M$ is a ring and $\phi, \phi_C, \phi_D$ are 
  ring homomorphisms, then for every $x \in X$ and $y \in Y$:
  \[
    \phi_C(\delta_x) \phi_D(\delta_y)  = \phi_D(\delta_y) \phi_C(\delta_x) = 0.
  \]
\end{corollary}
\begin{proof}
  Invoking Theorem~\ref{thm:nogo} in the case where 
  \[
  M_1 = (A *_{C(X)} \linfty(X) ) *_{C(Y)} \linfty(Y)
  \]
  is the colimit in $\Cstar$ of the diagram~\eqref{eq:nogodiagram} with $M$ deleted, we obtain
  that the images of $\delta_x$ and $\delta_y$ are orthogonal in $M_1$.
  Now let $R \circledast_S T$ denote the amalgamated free product of rings (which 
  coincides with the amalgamated free product of $\C$-algebras when $S$ is a unital subalgebra 
  of algebras $R$ and $T$), and let 
  \[
  M_0 = (A \circledast_{C(X)} \linfty(X) ) \circledast_{C(Y)} \linfty(Y)
  \] 
  be the colimit in the category of rings of the diagram~\eqref{eq:nogodiagram} with $M$ deleted. 
  There is a natural map $M_0 \to M_1$ induced by the universal property of $M_0$.
  It is a folk result that this is an embedding~\cite{blecherpaulsen:universal,ramseyreyes:amalgamated}.
  Thus the images of $\delta_x$ and $\delta_y$ in $M_0$ are already orthogonal.
  But the morphisms $\phi$, $\phi_C$, and $\phi_D$ of~\eqref{eq:nogodiagram} factor
  universally through $M_0$, so the images of $\delta_x$ and $\delta_y$ in $M$ are 
  orthogonal.
\end{proof}

We conclude this section with an obstruction for unbounded discretizations into topological algebras. 
Write $\cat{TAlg}$ for the category of Hausdorff topological $\C$-algebras with continuous homomorphisms. 
Recall~\cite[Chapter~10]{warner:topological} that a family $(a_i)_{i \in I}$ of elements in a Hausdorff topological ring $R$ is \emph{summable} if the net $(a_J)$ indexed by finite subsets $J \subseteq I$ converges, where $a_J=\sum_{j \in J} a_j$; in that case we write $\sum a_i$ for the limit.

\begin{corollary}\label{cor:topologicalnogo}
  Let $A$ be a C*-algebra with relatively diffuse C*-subalgebras $C \simeq C(X)$ and $D \simeq C(Y)$. 
  Then every unbounded $\cat{TAlg}$-discretization of $A$ is zero. More precisely: if there is a commutative diagram
    \[
    \begin{tikzpicture}[xscale=2.5,yscale=1.1]
      \node (CX) at (0,2) {\llap{$C \simeq$} $C(X)$};
      \node (A) at (0,1) {$A$};
      \node (CY) at (0,0) {\llap{$D \simeq$} $C(Y)$};
      \node (lX) at (1,2) {$\C^X$};
      \node (B) at (1,1) {$M$};
      \node (lY) at (1,0) {$\C^Y$};
      \draw[right hook->] (CX) to (A);
      \draw[left hook->] (CY) to (A);
      \draw[->] (A) to node[above] {$\phi$} (B);
      \draw[right hook->] (CX) to (lX);
      \draw[right hook->] (CY) to (lY);
      \draw[->] (lX) to node[right] {$\phi_C$} (B);
      \draw[->] (lY) to node[right] {$\phi_D$} (B);
    \end{tikzpicture}
  \]
  where $M$ is a Hausdorff topological ring, $\phi_C$ and $\phi_D$ are continuous homomorphisms, and $\phi$ is a homomorphism, then $M=0$.
\end{corollary}
\begin{proof}
  It suffices to prove the second, more general claim.
  Because the natural embedding $C(X) \hookrightarrow \C^X$ has image in the subring
  $\linfty(X) \subseteq \C^X$ and similarly for $C(Y)$, we may apply Corollary~\ref{cor:ringnogo} 
  to conclude that the idempotents $p_x = \phi_C(\delta_x)$ and $q_y = \phi_D(\delta_y)$ 
  satisfy $p_x q_y = 0$ for all $x \in X$ and $y \in Y$. 

  The orthogonal set of idempotents $\{\delta_x \mid x \in X\}$ is summable with $\sum \delta_x = 1$
  in $\C^X$, so the family of images $(p_x)_{x \in X}$ under the continuous homomorphism $\phi_C$
  is also summable in $M$ with $\sum p_x = 1$. Similarly, we have $(q_y)_{y \in Y}$ summable
  in $M$ with $\sum q_y = 1$.

  Now consider the net $(p_I q_J)$ indexed by the directed set of all `rectangular' subsets 
  $I \times J \subseteq X \times Y$ with both $I \subseteq X$ and $J \subseteq Y$ finite. 
  As both $(p_I)$ and $(q_J)$ converge to $1$, we have $p_I q_J \to 1^2 = 1$.
  But each $p_I q_J = \sum_I \sum_J p_x q_y = 0$, so we have $1 = \mathop{\mathrm{lim}} p_I q_J = 0$. 
  Thus $M = 0$.
\end{proof}

Just as in Example~\ref{ex:calkin}, if there is a morphism $B(H) \to A$ in $\Cstar$ with
$H$ an infinite-dimensional Hilbert space, then every unbounded $\cat{TAlg}$-discretization of
$A$ is trivial.

\begin{remark}
  Similar to the C*-discretization in Proposition~\ref{prop:pushoutdiscretization}, one could construct a pro-C*-discretization by replacing the pushouts
  $A *_C \linfty(\Spec(C))$ in $\Cstar$ with the pushouts $A *_C \C^{\Spec(C)}$ in
  $\proCstar$. However, the previous corollary shows that this construction must trivialize for algebras $A$ that have relatively diffuse commutative C*-subalgebras.
\end{remark}

We close with one further example of a separable algebra having no injective W*-discretizations.
We only sketch its proof, as the complete argument would require us to modify several results
above to account for possibly nonunital commutative subalgebras, a technicality that we have
avoided for the sake of readability.

\begin{example}
  Let $H = L^2[0,1]$ and $C=C[0,1] \subseteq \Linfty[0,1] \subseteq B(H)$. 
  Then $A=C+K(H)$ is a separable C*-algebra of type~I for which every AW*-discretization and 
  every unbounded $\TAlg$-discretizations has nonzero kernel.
  (It does, however, have nonzero non-injective such discretizations that factor through the 
  commutative C*-algebra $A/K(H)$.)
\end{example}
\begin{proof}
  Let $e_n$, and $q_n$ be as in Example~\ref{ex:kadisonsinger}. Within $B(H)$, write $C_0(\Z) \simeq D \subseteq K(H)$ 
  for the nonunital commutative C*-subalgebra generated by the $q_n$.
  If one alters Definition~\ref{def:relativelydiffuse} to allow for possibly nonunital C*-sub\-algebras, then $C$ and $D$ are relatively diffuse. Indeed, each pure state $\psi_0$ on $D$ is supported on some projection $p = q_n$, and every extension of $\psi_0$ to a state $\psi$ on $A$ satisfies $\psi(f) = \psi(pfp) = (\int_0^1 f \, dt) \psi(p) = \int_0^1 f \, dt$ for all $f \in C[0,1]$.
  A suitable modification of Theorem~\ref{thm:nogo} holds for such $C$ and $D$,
  with hardly a change to the proof. 
  
  Now if  $\phi \colon A \to M$ is an AW*-discretization or an unbounded 
  $\TAlg$-discre\-ti\-zation, then we claim that $K(H) \subseteq \ker(\phi)$. Indeed, the
  same method of proof of Corollaries~\ref{cor:awstarnogo} and~\ref{cor:topologicalnogo} 
  shows that $D$ is contained in $\ker(\phi)$ (noting that $C$ is still a unital
  subalgebra), and $K(H)$ is the ideal generated by $D$.
\end{proof}

\section{Conclusion}
\label{sec:conclusion}

In contrast to the obstructions~\cite{reyes:obstructing, bergheunen:extending, benzvimareyes:kochenspecker}, based on the Kochen-Specker theorem~\cite{kochenspecker:hiddenvariables} from quantum physics, the fact that profinite completion faithfully discretizes all finite-dimensional C*-algebras shows that the results in Section~\ref{sec:nogo} are truly infinite-dimensional obstructions and are therefore
independent of the Kochen-Specker theorem.

From the perspective of discretization as discussed in this paper, the search for a suitable 
candidate $\Alg$ for a category of algebras dual to `noncommutative sets' remains open. 
Having ruled out various candidates, we now briefly discuss the implications, including possible avenues to avoid these obstructions. 

Within the category $\Cstar$, there remains the interesting open Question~\ref{question} of whether 
every C*-algebra has a functorial (or equivalently, compatible) C*-discretization that is injective or faithful.
This question is addressed in recent work of Kornell~\cite{kornell:vstar} that takes a radically different approach: passing to a model of set theory in which every subset of $\mathbb{R}$ is measurable, so that the Axiom of Choice fails. 

A positive answer to Question~\ref{question} would still not entail a candidate category of algebras dual to `noncommutative sets'. That would require isolating a suitable subcategory $\Alg$ of $\Cstar$ containing the algebras $\linfty(X)$ and their normal $*$-homomorphisms as a full subcategory (dual to `classical' sets). 
One of the most notable feature of the algebras $\linfty(X)$ and $\C^X$ is their abundance of projections. But using this structure as a guide makes Corollaries~\ref{cor:awstarnogo} and~\ref{cor:topologicalnogo} particularly troubling.
Suppose that $A$, $C(X)$, and $C(Y)$ are as in Theorem~\ref{thm:nogo}.
Let $\phi \colon A \to M$ be the discretization of Proposition~\ref{prop:pushoutdiscretization}.
On the one hand, that proposition demonstrates that $\linfty(X)$ and $\linfty(Y)$ embed faithfully into $M$.
On the other hand, for all $x \in X$ and $y \in Y$, Theorem~\ref{thm:nogo} implies that the images of 
$\delta_x \in \linfty(X)$ and $\delta_y \in \linfty(Y)$ are orthogonal in $M$. 
So it is not contradictory to faithfully embed both $\linfty(X)$ and $\linfty(Y)$ into a common discretization making all $\delta_x \delta_y$ vanish.

Thus Corollaries~\ref{cor:awstarnogo} and~\ref{cor:topologicalnogo} merely indicate that globally `gluing' projections via the structure of an AW*-algebra or via convergence of nets of finite sums is inadequate for discretization.
This suggests exploring new structures imposing a suitable `global coherence' on projections in noncommutative $*$-algebras beyond AW*-algebras or topological algebras.
To speculate only about a single possibility: the notion of contramodule~\cite{positselski:contramodules} formalizes `infinite summation' operations that cannot be interpreted as convergence of sums in any topology.

\section*{Acknowledgements}
  C.~Heunen was supported by EPSRC Fellowship EP/L002388/1. 
  M.~Reyes was supported by NSF grant DMS-1407152. 
  We thank the anonymous referee for an insightful report that sharpened the statement of
  Theorem~\ref{thm:nogo}, informed us of Example~\ref{ex:circle}, and inspired us to significantly
  expand and rewrite this paper. 
%\end{acknowledgements}

\bibliographystyle{plain}
\bibliography{discretization-arxiv-v2}

\end{document}